\documentclass[12pt]{amsart}

\usepackage{amssymb}
\usepackage[top=2cm,bottom=5cm,left=3cm,right=1cm]{geometry}
\usepackage[hidelinks]{hyperref}

\theoremstyle{plain}
\newtheorem{theorem}{Theorem}[section]
\newtheorem{lemma}[theorem]{Lemma}
\newtheorem{proposition}[theorem]{Proposition}
\newtheorem{corollary}[theorem]{Corollary}

\theoremstyle{definition}

\newtheorem{example}[theorem]{Example}
\newtheorem{question}[theorem]{Question}

\theoremstyle{remark}
\newtheorem{remark}[theorem]{Remark}

\title{Finite Orders of Products of Two Class Transpositions Divide 840}
\author{Alex Iskra}
\address{Sobolev Institute of Mathematics, 4 Acad. Koptyug avenue, 630090, Novosibirsk, Russia.}
\address{Novosibirsk State  University, 2 Pirogova Street, 630090, Novosibirsk, Russia.} 
\email{a.iskra@g.nsu.ru}
\date{\today}

\begin{document}
	
	\begin{abstract}
		We prove that if the product of two class transpositions has finite
		order, then its order divides 840. This gives a complete affirmative
		answer to Question 18.48 in the Kourovka Notebook. We also obtain a criterion for the product of two class transpositions to have infinite order. In addition, we
		construct a pair of class transpositions whose product has order 840,
		showing that the bound is sharp.
	\end{abstract}
	
	\maketitle	
	
	
	\section{Introduction}
	
	In 2010, S.~Kohl introduced the notion of a class transposition of the integers in \cite{K}. A class transposition is an involution that interchanges the elements of two disjoint residue classes \(r_1(m_1)\) and \(r_2(m_2)\). A class transposition is denoted by \(\tau_{r_1(m_1),r_2(m_2)}\). The group generated by all class transpositions is denoted by \(\operatorname{CT}(\mathbb{Z})\).
	
	The cycle structure of products of two class transpositions
	was already studied by Kohl in a 2009 preprint, revised
	in 2011 \cite{KohlCycles}. This work identified
	\(18\) essentially different types of mutual intersections
	of the four residue classes involved and presented computations
	using the \textsf{RCWA} package for \textsf{GAP}.
	
	In 2012, at the Nikolaus Conference in Aachen, Kohl
	posed the question of the possible orders of products of two
	class transpositions, and in 2013 he posted it on
	MathOverflow \cite{KohlMO}. During the investigation, M.~Cuntz found
	examples of orders \(8\), \(24\), and \(42\), after which Kohl also found orders
	\(40\), \(84\), \(120\), \(168\), and \(420\).
	In 2014, the eighteenth edition of the Kourovka Notebook included
	Question~18.48 \cite{Kour}: \emph{is it true that there are only finitely many positive integers that occur as orders of products of two class transpositions?}
	
	For horizontal class transpositions, this question was answered in \cite{B}.
	A transposition is called \textit{horizontal} if the residue classes
	it interchanges have the same modulus. It was proved
	that the order of the product of two horizontal class
	transpositions belongs to the set \(\{1,2,3,4,6,12\}\), and that each of these values occurs. The same paper
	introduced the graph of a pair of class transpositions and established
	a connection between its connected components and the cycles of their product.
	
	The investigation was continued in a paper by the same authors
	on slanted class transpositions \cite{BI-slanted},
	published in 2026.
	In particular, it was shown that, for two transpositions
	with a common interchanged residue class, the order
	of the product is \(1\), \(3\), or \(\infty\), and a criterion
	for this order to be infinite was obtained. In addition,
	pairs with coinciding residues or moduli were studied,
	and the possible orders of their products were determined
	under additional conditions. This provided answers to Kohl's
	question for several special families of pairs.
	
	In this paper, Question~18.48 is answered for an arbitrary
	pair of class transpositions. We prove that
	\[
	\operatorname{ord}(ab)<\infty
	\quad\Longrightarrow\quad
	\operatorname{ord}(ab)\mid840.
	\]
	Consequently, the set of finite orders of such products
	is finite. Moreover, the bound is sharp: for the transpositions
	\(\tau_{0(40),\,18(28)}\) and
	\(\tau_{4(42),\,5(15)}\), the order of the product is
	\(840\). This equality was verified by Kohl using
	the \textsf{RCWA} package \cite{RCWA}.
	
	The proof develops the graph approach of \cite{B}.
	Segments of components of the graph \(\Gamma(a,b)\) are encoded by words
	over an eight-letter alphabet, and the existence of such segments
	is described by integer solutions of chains of linear
	Diophantine equations. Results on block repetition
	allow components of unbounded size to be constructed
	under suitable conditions. An analysis of short words leads
	to a criterion for infinite order: the product of class transpositions \(a\) and \(b\) has
	infinite order if and only if the pair
	\((a,b)\) exactly interprets some word of length~\(7\).
	It follows that, when the order is finite, all cycles
	of the permutation \(ab\) have length at most~\(8\), and hence
	its order divides
	\(\operatorname{lcm}(1,2,3,4,5,6,7,8)=840\).
	
	The paper is organized as follows. Section~\ref{sec:graph}
	introduces the graph of two class transpositions and describes the connection
	between its components and the cycles of the product. Section~\ref{sec:words}
	introduces the eight-letter language, the Diophantine interpretation of words,
	and the transformations defining their orbits. Section~\ref{sec:pumping}
	proves lemmas on the repetition of blocks of Diophantine equations
	and their corollaries concerning the interpretability of powers of words.
	Section~\ref{sec:infinite-order} establishes sufficient conditions
	for the product to have infinite order, involving the exact interpretability
	of squares, almost squares, and a special five-letter word.
	Section~\ref{sec:criterion} proves that four six-letter words
	are non-interpretable and establishes a criterion for infinite order
	by analyzing factors of words of length~\(7\). Finally,
	Section~\ref{sec:main-result} proves the main theorem
	and gives an example of a product of order~\(840\), showing
	that the bound is sharp.
		
	\section{The graph of two class transpositions}
	\label{sec:graph}
	
	Let \(r,m\in\mathbb Z\) and $0\leq r<m$. The class of integers congruent to \(r\) modulo \(m\) is denoted
	by \(r(m)\):
	\[
	r(m)=\{r+mk\mid k\in\mathbb Z\}.
	\]
	
	If \(r_1(m_1)\cap r_2(m_2)=\varnothing\), then the \textit{class transposition}
	\(\tau_{r_1(m_1),r_2(m_2)}\) is defined to be the involution interchanging
	the integers \(r_1+km_1\) and \(r_2+km_2\) for each \(k\in\mathbb Z\)
	and fixing all other integers. Thus,
	\[
	\tau_{r_1(m_1),r_2(m_2)}
	=
	\prod_{k\in\mathbb Z}
	(r_1+m_1k,r_2+m_2k) \in \operatorname{Sym}(\mathbb{Z}).
	\]
	Note that $r_1(m_1)\cap r_2(m_2)=\varnothing$ if and only if $r_1\not\equiv r_2\pmod{\gcd(m_1,m_2)}$.
	
	Fix two class transpositions
	\[
	a=\tau_{r_1(m_1),r_2(m_2)}
	=
	\prod_{k\in\mathbb Z}
	(r_1+m_1k,r_2+m_2k),
	\]
	\[
	b=\tau_{r_3(m_3),r_4(m_4)}
	=
	\prod_{l\in\mathbb Z}
	(r_3+m_3l,r_4+m_4l).
	\]
	Consider the disjoint sets
	\[
	V_1=\{a_k,b_k\mid k\in\mathbb Z\},
	\quad
	V_2=\{c_l,d_l\mid l\in\mathbb Z\}.
	\]
	
	Define the \textit{graph of two class transpositions} \(\Gamma(a,b)=(V,E)\), where the vertex set is given by \(V=V_1\sqcup V_2\). Define a map \(\mu\colon V\longrightarrow \mathbb{Z}\) by
	\[
	\mu(a_k)=r_1+m_1k,
	\quad
	\mu(b_k)=r_2+m_2k,
	\quad
	\mu(c_l)=r_3+m_3l,
	\quad
	\mu(d_l)=r_4+m_4l.
	\]
	
	The set \(E\) consists of all unordered pairs
	\(\{v,w\}\subseteq V\) satisfying one of the following
	conditions:
	\begin{enumerate}
		\item \(\mu(v)=\mu(w), \quad v\in V_1,\quad w\in V_2\);
		\item \(a(\mu(v))=\mu(w), \quad v,w\in V_1\);
		\item \(b(\mu(v))=\mu(w), \quad v,w\in V_2\).
	\end{enumerate}
	Edges satisfying the first condition are called \textit{edges of the first
	type}, and edges satisfying the second or third condition are called
	\textit{edges of the second type}.
	
	Paths are allowed to have repeated vertices and edges.
	A path is called \textit{simple} if its vertices are
	pairwise distinct.
	
	\begin{lemma}[\cite{B}]
		\label{lem:vertex-degrees}
		For each vertex \(v\) of the graph \(\Gamma(a,b)\),
		exactly one of the following conditions holds:
		\begin{enumerate}
			\item the vertex \(v\) has valency \(1\) and is incident to
			an edge of the second type;
			\item the vertex \(v\) has valency \(2\) and is incident to
			one edge of the first type and one edge of the second type.
		\end{enumerate}
	\end{lemma}
	
	Consequently, each connected component of the graph \(\Gamma(a,b)\)
	is a finite cycle, a finite path, a two-way infinite
	path, or a one-way infinite path.
	
	The \textit{length} of a connected component \(C\) is defined to be the number of its edges
	and is denoted by \(|C|\).
	
	\begin{theorem}[\cite{B}]
		\label{propmain}
		Let \(\{C_i\}_{i\in I}\) be the set of connected components
		of the graph \(\Gamma(a,b)\). For each \(i\in I\), set
		\[
		\nu_i=
		\begin{cases}
			\dfrac{|C_i|}{4},
			& \text{if \(C_i\) is a finite cycle},\\[2mm]
			\dfrac{|C_i|+3}{2},
			& \text{if \(C_i\) is a finite path},\\[2mm]
			\infty,
			& \text{if \(C_i\) is infinite}.
		\end{cases}
		\]
		Then $\operatorname{ord}(ab) = \operatorname{lcm}\{\nu_i\mid i\in I\}$.
	\end{theorem}
	
	Here the least common multiple is taken to be \(\infty\) if
	the family \(\{\nu_i\}_{i\in I}\) contains \(\infty\) or has no
	finite common multiple. The length of each finite cycle is divisible
	by \(4\), and the length of each finite path is odd, so the numbers
	\(\nu_i\) in the first two cases are integers.
	
	\begin{remark}
		\label{rem:component-cycles}
		The order formula in Theorem~\ref{propmain} follows
		from the following correspondence established in \cite{B}.
		
		Each component \(C_i\) corresponds to one cycle of \(ab\)
		of length \(\nu_i\) if it is a finite path or a one-way
		infinite path, and to two disjoint cycles of \(ab\),
		each of length \(\nu_i\), if it is a finite cycle or
		a two-way infinite path.
		Cycles corresponding to distinct connected components
		are disjoint.
	\end{remark}
	
	Thus, the order of the product \(ab\) is determined by the types and
	lengths of the connected components of the graph \(\Gamma(a,b)\).
	
	\section{The eight-letter language and Diophantine interpretation}
	\label{sec:words}
	
	\subsection{Structurally admissible words and exact interpretability}
	
	Set 
	$$
	\Sigma = \{\ell,r,p,q,\bar\ell,\bar r,\bar p,\bar q\}.
	$$
	Associate with each letter \(\alpha \in \Sigma\) a four-letter word \(\pi(\alpha)\):
	\[
	\begin{array}{c|cccccccc}
		\alpha
		&\ell&r&p&q&\bar\ell&\bar r&\bar p&\bar q\\ \hline
		\pi(\alpha)
		&BACD&ABDC&BADC&ABCD&DCAB&CDBA&CDAB&DCBA.
	\end{array}
	\]
	
	A word \(w=\alpha_1\alpha_2\cdots\alpha_n\) is called \textit{structurally admissible} if \(n\geq1\) and the word formed by the last two letters of \(\pi(\alpha_i)\) coincides with the word formed by the first two letters of \(\pi(\alpha_{i+1})\) for each \(1\leq i<n\). For example, the word \(\ell\bar r\) is structurally admissible.
	
	\begin{remark}
		In every structurally admissible word, barred and unbarred
		letters alternate.
	\end{remark}
	
	For a letter \(\alpha\in\Sigma\), let \(\chi(\alpha)\) denote the pair
	formed by the second and third symbols of the word \(\pi(\alpha)\).
	Thus,
	\[
	\begin{array}{c|cccccccc}
		\alpha&\ell&r&p&q&\bar\ell&\bar r&\bar p&\bar q\\ \hline
		\chi(\alpha)&AC&BD&AD&BC&CA&DB&DA&CB.
	\end{array}
	\]
	
	In what follows, the order of the classes in each of the two class transpositions \(a\) and \(b\) is fixed. Choose the parameters \(r_X\) and \(m_X\), \(X\in\{A,B,C,D\}\), according to the table
	\[
	\begin{array}{c|cccc}
		X&A&B&C&D\\ \hline
		r_X&r_1&r_2&r_3&r_4\\
		m_X&m_1&m_2&m_3&m_4
	\end{array}
	\]
	
	Let \(\chi(\alpha) = XY\). Associate with the letter \(\alpha\) the linear Diophantine equation
	\[
	r_X+m_Xu=r_Y+m_Yv,
	\]
	which we denote by \(E_{\alpha}(u,v)\), or, equivalently,
	\[
	m_Xu-m_Yv=r_Y-r_X.
	\]
	Let \(w=\alpha_1\alpha_2\cdots \alpha_n\) be a structurally admissible word. Its \textit{Diophantine system} \(\mathcal S(w)\)
	is defined to be the system
	\[
	E_{\alpha_1}(u_1, u_2), E_{\alpha_2}(u_2, u_3), \ldots, E_{\alpha_n}(u_n, u_{n+1}).
	\]
	A word \(w\) is called \textit{interpretable} by a pair of class transpositions \((a, b)\) if the system \(\mathcal S(w)\) has an integer solution.
	
	A word \(w\) is called \textit{exactly interpretable} by a pair \((a,b)\) if the system \(\mathcal S(w)\) has an integer solution \((u_1,\ldots,u_{n+1})\) such that \(u_i\ne u_j\) for all \(1\leq i<j\leq n+1\) with \(i\equiv j\pmod 2\).
	
	Label the vertices \(a_k, b_k, c_k, d_k\) of the graph \(\Gamma(a, b)\) by the symbols \(A, B, C, D\), respectively.
	
	\begin{proposition}
		\label{prop:word-path-correspondence}
		Let \(w=\alpha_1\cdots\alpha_n\) be a structurally admissible
		word. Each integer solution of the system \(\mathcal S(w)\)
		determines an oriented path
		\[
		v_1,v_2,\ldots,v_{2n+2}
		\]
		in the graph \(\Gamma(a,b)\), whose first and last edges
		are of the second type, and such that, for each \(1\leq i\leq n\),
		the sequence of labels of the vertices
		\[
		v_{2i-1},v_{2i},v_{2i+1},v_{2i+2}
		\]
		coincides with the word \(\pi(\alpha_i)\).
		
		Conversely, each such oriented path uniquely
		determines a structurally admissible word and an integer solution
		of its Diophantine system. The path is simple
		if and only if the corresponding interpretation is exact.
	\end{proposition}
	
	\begin{proof}
		The structural admissibility of the word \(w\) implies the equalities
		\[
		\pi(\alpha_i)=X_iY_iX_{i+1}Y_{i+1},
		\qquad
		X_iY_i\in\{AB,BA,CD,DC\}.
		\]
		The pair \(X_iY_i\) and the parameter \(u_i\) correspond to
		an oriented edge of the second type with parameter \(u_i\).
		The equation \(E_{\alpha_i}(u_i,u_{i+1})\) means that the terminal vertex
		of this edge is joined by an edge of the first type to the initial vertex of the next
		edge of the second type. This gives the required path.
		
		Conversely, the labels of consecutive edges of the second type and
		of the edge of the first type joining them uniquely determine the letter
		\(\alpha_i\) and the equation \(E_{\alpha_i}(u_i,u_{i+1})\).
		
		Exactness of the interpretation means that edges of the second type do not
		repeat. Then edges of the first type do not repeat either, since
		each vertex is incident to exactly one edge of the second type.
		In a graph of valency at most \(2\), such a path is simple:
		it cannot be a cycle, since it contains \(n+1\) edges of the second
		type and \(n\) edges of the first type. The converse is obvious.
	\end{proof}
	
	\begin{lemma}
		\label{lem:factor}
		If a word \(w\) is exactly interpretable by a pair of class
		transpositions \((a,b)\), then every nonempty factor of \(w\) is exactly
		interpretable by the same pair.
	\end{lemma}
	
	\begin{proof}
		By Proposition~\ref{prop:word-path-correspondence}, an exact
		interpretation of the word \(w\) determines a simple oriented path in
		the graph \(\Gamma(a,b)\). A factor of the word corresponds to a segment of this
		path. This segment is also a simple oriented path, so,
		by the same proposition, it determines an exact interpretation of the factor.
	\end{proof}
	
	\subsection{Orbits of words}
	
	Consider the following permutations of the alphabet \(\Sigma\):
	\[
	\begin{aligned}
		\sigma_1&=
		(\ell\,q)(r\,p)
		(\bar\ell\,\bar q)(\bar r\,\bar p),\\
		\sigma_2&=
		(\ell\,p)(r\,q)
		(\bar\ell\,\bar p)(\bar r\,\bar q),\\
		\sigma_3&=
		(\ell\,\bar\ell)(r\,\bar r)
		(p\,\bar q)(q\,\bar p).
	\end{aligned}
	\]
	The permutation \(\sigma_1\) corresponds to changing the order of the residue classes in the notation for the transposition \(a\), the permutation \(\sigma_2\) to
	changing the order of the residue classes in the notation for \(b\), and the permutation
	\(\sigma_3\) to replacing the pair \((a,b)\) by the pair \((b,a)\).
	
	For example, if the transposition \(a=\tau_{r_1(m_1),r_2(m_2)}\) is written as \(a=\tau_{r_2(m_2),r_1(m_1)}\), leaving the notation for the transposition \(b\) unchanged, then the labels \(A\)
	and \(B\) are interchanged. Consequently, each letter
	\(\alpha\) of the corresponding word is replaced by the letter
	\(\sigma_1(\alpha)\).
	
	Extend each permutation \(\sigma_i\) to a permutation
	of the set of words over the alphabet \(\Sigma\), acting letter by letter:
	\[
	\sigma_i(\alpha_1\cdots\alpha_n)
	=
	\sigma_i(\alpha_1)\cdots\sigma_i(\alpha_n),
	\qquad i\in\{1,2,3\}.
	\]
	
	For \(\alpha\in\Sigma\), let \(\bar\alpha\) denote the letter
	obtained by adding or removing the bar. Set
	\[
	\rho(\alpha_1\cdots\alpha_n)
	=
	\bar\alpha_n\cdots\bar\alpha_1.
	\]
	The transformation \(\rho\) corresponds to traversing the oriented
	path in the opposite direction. All four transformations
	are involutions and preserve the structural admissibility of words.
	
	Let \(\operatorname{Orb}(w)\) denote the orbit of the word \(w\)
	under the action of the group \(\langle\sigma_1,\sigma_2,\sigma_3,\rho\rangle\).
	
	\begin{remark}
		\label{rem:orbit-invariance}
		Let the words \(w\) and \(\widetilde w\) belong to the same orbit.
		If \(w\) is interpretable by a pair of class transpositions \((a,b)\),
		then \(\widetilde w\) is interpretable by some pair
		\((\widetilde a,\widetilde b)\). Moreover, exactness of the interpretation
		and the order of the product of the transpositions are preserved.
		
		Indeed, it suffices to check this for the generating
		transformations. Applying \(\sigma_1\) or \(\sigma_2\) changes
		only the order of the classes in the notation for the transpositions and hence does not change
		the pair \((a,b)\) itself. Applying \(\sigma_3\) replaces it by the pair
		\((b,a)\), while applying \(\rho\) reverses the direction of traversal
		of the corresponding path without changing the interpreting pair.
		All these operations preserve the simplicity of the path and hence
		also the exactness of the interpretation.
		
		Finally, the only operation that changes the product replaces
		\(ab\) by \(ba\). Since \(a^2=1\), we have $ba=a(ab)a$, so \(\operatorname{ord}(ba)=\operatorname{ord}(ab)\). Thus, when considering the interpretability of words,
		it suffices to choose one representative from each orbit.
	\end{remark}
	
	\section{Repetition of blocks of Diophantine equations}
	\label{sec:pumping}
	
	\begin{lemma}
		\label{lem:constant-chain}
		Let \(P,Q\) be positive integers, let \(H\in\mathbb Z\),
		and suppose that $\gcd(P,Q)\mid H$. Then, for each \(N\geq1\), the system
		\[
		\begin{cases}
			Pt_0-Qt_1=H,\\
			Pt_1-Qt_2=H,\\
			\vdots\\
			Pt_{N-1}-Qt_N=H
		\end{cases}
		\]
		has an integer solution.
	\end{lemma}
	
	\begin{proof}
		Set
		\[
		d=\gcd(P,Q),\qquad p=P/d,\qquad q=Q/d,\qquad h=H/d.
		\]
		Then \(p,q,h\) are integers, \(\gcd(p,q)=1\), and it suffices to solve the system
		\[
		pt_i-qt_{i+1}=h,\qquad 0\le i<N.
		\]
		
		We proceed by induction on \(N\). For \(N=1\), the existence of a solution
		is obvious.
		
		Let \((t_0^*,\ldots,t_N^*)\) be an integer solution of the system
		of length \(N\). For an arbitrary \(z\in\mathbb Z\), set
		\[
		\widetilde t_i=t_i^*+p^i q^{N-i}z,
		\qquad 0\le i\le N.
		\]
		This substitution preserves all the equalities already satisfied, since,
		for \(0\le i<N\), we have
		\[
		p\widetilde t_i-q\widetilde t_{i+1}
		=
		pt_i^*-qt_{i+1}^*
		+\bigl(p^{i+1}q^{N-i}-p^{i+1}q^{N-i}\bigr)z
		=h.
		\]
		
		Since \(\gcd(p^{N+1},q)=1\), we can choose \(z\in\mathbb Z\)
		such that
		\[
		p^{N+1}z\equiv h-pt_N^*\pmod q.
		\]
		With this choice, the number
		\[
		\widetilde t_{N+1}
		=
		\frac{p\widetilde t_N-h}{q}
		=
		\frac{pt_N^*+p^{N+1}z-h}{q}
		\]
		is an integer and satisfies the equality
		\[
		p\widetilde t_N-q\widetilde t_{N+1}=h.
		\]
		Consequently,
		\((\widetilde t_0,\ldots,\widetilde t_{N+1})\)
		is an integer solution of the system of length $N+1$.
		This completes the induction.
	\end{proof}
	
	\begin{lemma}
		\label{lem:block}
		Let \(\mathcal A(x_1,\ldots,x_m)\) be a system of linear
		Diophantine equations in the variables \(x_1,\ldots,x_m\)
		of the form
		\[
		a_i x_i-b_i x_{i+1}=c_i,
		\qquad 1\le i<m,
		\]
		where \(a_i,b_i\) are positive integers and \(c_i\in\mathbb Z\).
		Let \(A,B\) be positive integers and let \(C\in\mathbb Z\).
		
		If the system
		\[
		\begin{cases}
			\mathcal A(x_1,\ldots,x_m),\\
			Ax_m-By_1=C,\\
			\mathcal A(y_1,\ldots,y_m)
		\end{cases}
		\]
		has an integer solution, then, for each \(N\ge1\), the system
		\[
		\begin{cases}
			\mathcal A\bigl(x_1^{(j)},\ldots,x_m^{(j)}\bigr),
			& 1\le j\le N,\\[1mm]
			A x_m^{(j)}-B x_1^{(j+1)}=C,
			& 1\le j<N
		\end{cases}
		\]
		also has an integer solution.
	\end{lemma}
	
	\begin{proof}
		Let \(x^*=(x_1^*,\ldots,x_m^*)\) and
		\(y^*=(y_1^*,\ldots,y_m^*)\) be integer solutions
		of the system \(\mathcal A\) such that
		\[
		Ax_m^*-By_1^*=C.
		\]
		For \(N=1\), it suffices to take the solution \(x^*\). If \(x^*=y^*\),
		then, for any \(N\), we can set \(x_i^{(j)}=x_i^*\).
		
		Suppose that \(x^*\ne y^*\). Subtracting the corresponding equations, we obtain
		\[
		a_i(y_i^*-x_i^*)=b_i(y_{i+1}^*-x_{i+1}^*),
		\qquad 1\le i<m.
		\]
		Since \(a_i,b_i>0\), all the differences \(y_i^*-x_i^*\) are nonzero
		and have the same sign. Therefore,
		\[
		y_i^*=x_i^*+\varepsilon q_i,
		\qquad
		\varepsilon\in\{-1,1\},\quad q_i\in\mathbb Z_{>0},
		\]
		where \(a_iq_i=b_iq_{i+1}\). Consequently,
		\[
		x_i(t)=x_i^*+tq_i,\qquad 1\le i\le m,
		\]
		defines an integer solution of the system \(\mathcal A\)
		for each \(t\in\mathbb Z\).
		
		Set
		\[
		P=Aq_m,\qquad Q=Bq_1,\qquad
		H=C-Ax_m^*+Bx_1^*=-\varepsilon Q.
		\]
		Then \(P,Q>0\) and \(\gcd(P,Q)\mid H\).
		By Lemma~\ref{lem:constant-chain}, for each \(N\ge2\),
		there exist integers \(t_1,\ldots,t_N\) satisfying
		the equalities
		\[
		Pt_j-Qt_{j+1}=H,\qquad 1\le j<N.
		\]
		Set
		\[
		x_i^{(j)}=x_i^*+q_it_j,
		\qquad 1\le i\le m,\quad 1\le j\le N.
		\]
		Each tuple \((x_1^{(j)},\ldots,x_m^{(j)})\) satisfies
		the system \(\mathcal A\), and, for \(1\le j<N\), we have
		\[
		Ax_m^{(j)}-Bx_1^{(j+1)}
		=Ax_m^*-Bx_1^*+Pt_j-Qt_{j+1}
		=Ax_m^*-Bx_1^*+H=C.
		\]
		This gives the required integer solution.
	\end{proof}
	\begin{corollary}
		\label{cor:almost-square-powers}
		Let \(U=\alpha_1\cdots\alpha_s\), where \(s\ge2\), and
		\(\widehat U=\alpha_1\cdots\alpha_{s-1}\).
		If the word \(U\widehat U\) is interpretable by a pair of class
		transpositions \((a,b)\), then the word \(U^N\) is interpretable by this pair
		for each \(N\ge1\).
	\end{corollary}
	
	\begin{proof}
		All transitions between adjacent letters of the words \(U^N\widehat U\)
		occur in the word \(U\widehat U\), so these words
		are structurally admissible.
		
		Write \(U=\widehat U\alpha_s\). The system
		\(\mathcal S(U\widehat U)\) consists of two copies of
		\(\mathcal S(\widehat U)\) joined by the equation
		\(E_{\alpha_s}\). By assumption, this system has an integer
		solution.
		
		Fix \(N\ge1\). By Lemma~\ref{lem:block}, we can
		join \(N+1\) copies of \(\mathcal S(\widehat U)\)
		by the same equation so that the resulting system has
		an integer solution. This system corresponds to the word
		\[
		(\widehat U\alpha_s)^N\widehat U=U^N\widehat U.
		\]
		Restricting the resulting solution to the first \(Ns+1\) variables,
		we obtain an integer solution of the system \(\mathcal S(U^N)\).
		Consequently, the word \(U^N\) is interpretable by the pair \((a,b)\).
	\end{proof}
	
	\begin{corollary}
		\label{cor:block-powers}
		If the word \(U^2\) is interpretable by the pair \((a, b)\), then the word \(U^N\) is interpretable by this pair for each \(N\geq1\).
	\end{corollary}
	
	\begin{proof}
		The structural admissibility of \(U^2\) implies that the length of \(U\)
		is even and at least \(2\). Since \(U\widehat U\) is
		a prefix of \(U^2\), it is also interpretable by the pair \((a,b)\).
		It remains to apply Corollary~\ref{cor:almost-square-powers}.
	\end{proof}
	
	\section{Sufficient conditions for infinite order}
	\label{sec:infinite-order}
	
	\begin{lemma}
		\label{lem:component-pumping}
		Let \((a,b)\) be a pair of class transpositions, let \(U=\alpha_1\alpha_2\cdots\alpha_s\), where \(s\geq1\), and suppose that the word \(U^2\) is structurally admissible. Assume that the following
		conditions hold:
		\begin{enumerate}
			\item the system \(\mathcal S(U^2)\) has an integer solution;
			\item the system \(\mathcal S(U)\) has an integer solution \((v_1^*,\ldots,v_{s+1}^*)\) such that \(v_1^*\ne v_{s+1}^*\).
		\end{enumerate}
		Then \(\operatorname{ord}(ab)=\infty\).
		
	\end{lemma}
	
	\begin{proof}
		Let \(\chi(\alpha_i)=X_iY_i\). The equation of the system \(\mathcal S(U)\) corresponding to the letter
		\(\alpha_i\) can be written as
		\[
		u_{i+1}=f_i(u_i),
		\qquad 1\leq i\leq s,
		\]
		where
		\[
		f_i(x)
		=
		\frac{m_{X_i}}{m_{Y_i}}x
		+
		\frac{r_{X_i}-r_{Y_i}}{m_{Y_i}}.
		\]
		Set
		\[
		F(x)=(f_s\circ\cdots\circ f_1)(x)
		=\lambda x+\beta,\qquad \lambda>0.
		\]
		Then every solution of the system \(\mathcal S(U)\)
		satisfies \(u_{s+1}=F(u_1)\). In particular, the condition \(v_1^*\ne v_{s+1}^*\) implies that \(F(v_1^*)=v_{s+1}^*\ne v_1^*\). Therefore,
		\begin{equation}
			F\ne\operatorname{id}_{\mathbb Q}.
			\label{eq:pumping-nonidentity}
		\end{equation}
		
		The structural admissibility of the word \(U^2\) implies the structural admissibility of the word \(U^N\) for each \(N\geq1\).
		
		The solvability of the system \(\mathcal S(U^2)\) and
		Corollary~\ref{cor:block-powers}
		imply that the system \(\mathcal S(U^N)\) is solvable over the
		integers for each \(N\geq2\).
		
		Fix \(N\geq2\) and choose an integer solution \(u^*=(u_1^*,\ldots,u_{Ns+1}^*)\) of the system \(\mathcal S(U^N)\) and a positive integer solution \(q=(q_1,\ldots,q_{Ns+1})\) of the homogeneous system
		corresponding to \(\mathcal S(U^N)\).
		Consequently,
		\begin{equation}
			u_i(t)=u_i^*+tq_i,
			\qquad t\in\mathbb Z,
			\label{eq:pumping-solutions}
		\end{equation}
		defines an infinite family of integer solutions of the system
		\(\mathcal S(U^N)\).
		
		Since \(F\) is a nonidentity map
		with positive coefficient \(\lambda\), it is strictly increasing and has
		at most one fixed point. Since \(q_1>0\), the quantity
		\(u_1(t)\) takes infinitely many values. We can therefore
		choose \(t\in\mathbb Z\) such that
		\begin{equation}
			F(u_1(t))\ne u_1(t).
			\label{eq:pumping-nonfixed}
		\end{equation}
		
		Note that
		\begin{equation}
			u_{ks+1}(t)=F^k(u_1(t)),
			\qquad 0\leq k\leq N.
			\label{eq:pumping-iterates}
		\end{equation}
		Since \(F\) is strictly increasing and \(F(u_1(t))\ne u_1(t)\), the sequence \(F^k(u_1(t))\) is strictly monotone. Hence its terms are pairwise distinct.
		
		The structural admissibility of the word \(U^2\) means that \(s\) is even.
		Therefore, the \(N+1\) distinct parameters obtained determine \(N+1\)
		distinct edges of the second type in one component of the graph \(\Gamma(a,b)\).
		Since \(N\) is arbitrary, Theorem~\ref{propmain} gives
		\(\operatorname{ord}(ab)=\infty\).
	\end{proof}
	
	\begin{corollary}
		\label{cor:square-word}
		If the word \(U^2\) is exactly
		interpretable by a pair of class transpositions \((a,b)\), then \(\operatorname{ord}(ab)=\infty\).
	\end{corollary}
	
	\begin{proof}
		Set \(s=|U|\) and choose an integer solution \((u_1^*,\ldots,u_{2s+1}^*)\) of the system \(\mathcal S(U^2)\) that gives an exact interpretation
		of the word \(U^2\).
		
		In a structurally admissible word, barred and unbarred letters
		alternate. The first letter of the word \(U\) occurs in \(U^2\)
		at positions \(1\) and \(s+1\), so \(s\) is even.
		
		The first \(s\) equations of the system \(\mathcal S(U^2)\) form
		the system \(\mathcal S(U)\). Hence,
		\((u_1^*,\ldots,u_{s+1}^*)\) is an integer solution of this system.
		The indices \(1\) and \(s+1\) have the same parity,
		so exactness of the interpretation gives \(u_1^*\ne u_{s+1}^*\). Both conditions of Lemma~\ref{lem:component-pumping} are satisfied,
		and hence \(\operatorname{ord}(ab)=\infty\).
	\end{proof}
	
	\begin{lemma}
		\label{lem:two-letter-words}
		If the words \(\ell \bar r\) and \(\bar\ell r\) are interpretable
		by the same pair of class transpositions \((a,b)\), and
		at least one of them is exactly interpretable by this pair, then \(\operatorname{ord}(ab)=\infty\).
	\end{lemma}
	
	\begin{proof}
		We may assume that the word \(\ell\bar r\) is exactly interpretable and that
		the word \(\bar\ell r\) is interpretable. Otherwise, apply the permutation \(\sigma_3\), which interchanges these two words and
		simultaneously interchanges the transpositions.
		
		Choose integer solutions \((u_1^*,u_2^*,u_3^*)\) and \((v_1^*,v_2^*,v_3^*)\) of the systems \(\mathcal{S}(\ell\bar r)\) and \(\mathcal{S}(\bar\ell r)\), respectively, choosing the first solution to give an exact interpretation.
		They satisfy the equalities
		\begin{equation}
			\begin{cases}
				m_1u_1^*-m_3u_2^*=r_3-r_1,\\
				m_4u_2^*-m_2u_3^*=r_2-r_4,\\
				m_3v_1^*-m_1v_2^*=r_1-r_3,\\
				m_2v_2^*-m_4v_3^*=r_4-r_2.
			\end{cases}
			\label{eq:two-word-system}
		\end{equation}
		Set \(U=\ell\bar r\). Thus, \((u_1^*,u_2^*,u_3^*)\) is an integer
		solution of the system \(\mathcal S(U)\) whose first and last
		coordinates are distinct.
		
		We show that the system \(\mathcal S(\ell\bar r\ell)\) has
		an integer solution. Set
		\[
		M_1=\frac{m_1}{\gcd(m_1,m_3)},
		\quad
		M_3 = \frac{m_3}{\gcd(m_1,m_3)}.
		\]
		Then \(\gcd(M_1,M_3)=1\), and all integer solutions
		of the equation
		\[
		m_1x-m_3y=r_3-r_1
		\]
		have the form
		\begin{equation}
			x=u_1^*+M_3 t,
			\qquad
			y=u_2^*+M_1 t,
			\qquad t\in\mathbb Z.
			\label{eq:two-word-parametrization}
		\end{equation}
		
		The system \(\mathcal S(\ell\bar r\ell)\) has an integer
		solution if and only if there exist
		\(t,v\in\mathbb Z\) such that
		\[
		m_4(u_2^*+M_1 t)-m_2(u_1^*+M_3 v)=r_2-r_4.
		\]
		Using the second equation of the system~\eqref{eq:two-word-system}, this condition
		can be written as
		\begin{equation}
			Pt-Qv=H,
			\label{eq:two-word-connecting-equation}
		\end{equation}
		where
		\[
		P=M_1 m_4,
		\qquad
		Q=M_3 m_2,
		\qquad
		H=m_2(u_1^*-u_3^*).
		\]
		
		We prove that \(\gcd(P,Q)\mid H\). Note that \(\gcd(P,Q)\) divides the numbers \(M_1 \operatorname{lcm}(m_2,m_4)\) and \(M_3 \operatorname{lcm}(m_2,m_4)\). Since \(\gcd(M_1, M_3)=1\), we obtain
		\begin{equation}
			\gcd(P,Q) \mid \operatorname{lcm}(m_2,m_4).
			\label{eq:two-word-lcm-divisibility}
		\end{equation}
		
		Comparing the first and third equations of the system~\eqref{eq:two-word-system},
		we obtain
		\[
		m_1(v_2^*-u_1^*)=m_3(v_1^*-u_2^*).
		\]
		Consequently,
		\begin{equation}
			Q\mid m_2(v_2^*-u_1^*).
			\label{eq:two-word-first-divisibility}
		\end{equation}
		
		Comparing the second and fourth equations of the system~\eqref{eq:two-word-system},
		we obtain
		\[
		m_4(v_3^*-u_2^*)=m_2(v_2^*-u_3^*).
		\]
		The number on the right-hand side is divisible by both \(m_2\) and \(m_4\).
		Therefore,
		\begin{equation}
			\operatorname{lcm}(m_2,m_4)\mid m_2(v_2^*-u_3^*).
			\label{eq:two-word-second-divisibility}
		\end{equation}
		Finally,
		\[
		\begin{aligned}
			H
			&=m_2(u_1^*-u_3^*)\\
			&=m_2(v_2^*-u_3^*)-m_2(v_2^*-u_1^*).
		\end{aligned}
		\]
		It follows from~\eqref{eq:two-word-lcm-divisibility}--\eqref{eq:two-word-second-divisibility} that \(\gcd(P,Q)\mid H\).
		
		Consequently, equation~\eqref{eq:two-word-connecting-equation} has an integer
		solution, and hence the system
		\(\mathcal S(\ell\bar r\ell)\) is solvable over the integers.
		
		By Lemma~\ref{lem:block}, we obtain an integer
		solution of the system for the word \(\ell\bar r\ell\bar r\ell\). Restricting it to the first four equations, we conclude that
		the system \(\mathcal S(U^2) = \mathcal S(\ell\bar r\ell\bar r)\) has an integer solution.
		
		The word \(U^2\) is structurally admissible, and the system
		\(\mathcal S(U)\) has an integer solution
		\((u_1^*,u_2^*,u_3^*)\) with \(u_1^*\ne u_3^*\). Therefore, by
		Lemma~\ref{lem:component-pumping}, \(\operatorname{ord}(ab)=\infty\).
	\end{proof}
	
	\begin{corollary}
		\label{cor:five-letter-word}
		If the word \(\ell\bar r p\bar\ell r\) is exactly interpretable by a pair of class transpositions \((a,b)\), then \(\operatorname{ord}(ab)=\infty\).
	\end{corollary}
	
	\begin{proof}
		By Lemma~\ref{lem:factor}, the factors \(\ell\bar r\)
		and \(\bar\ell r\) are exactly interpretable by the same pair.
		The claim follows from Lemma~\ref{lem:two-letter-words}.
	\end{proof}
	
	\begin{lemma}
		\label{lem:almost-square}
		Let \(U=\alpha_1\alpha_2\cdots\alpha_s\), where \(s\geq2\), and let \(\widehat U=\alpha_1\alpha_2\cdots\alpha_{s-1}\). If the word \(U\widehat U\) is exactly interpretable by a
		pair of class transpositions \((a,b)\), then \(\operatorname{ord}(ab)=\infty\).
	\end{lemma}
	
	\begin{proof}
		Let \((u_1^*,u_2^*,\ldots,u_{2s}^*)\) be an integer solution of the system
		\(\mathcal S(U\widehat U)\) that gives an exact interpretation. Note that \(s\) is even.
		
		The first \(s\) equations form the system \(\mathcal S(U)\).
		Therefore, \((u_1^*,\ldots,u_{s+1}^*)\) is an integer solution of this system. The indices \(1\) and \(s+1\)
		have the same parity, so exactness of the original
		interpretation implies that \(u_1^*\ne u_{s+1}^*\).
		
		By Corollary~\ref{cor:almost-square-powers} with \(N = 2\), the system
		\(\mathcal S(U^2)\) has an integer solution. Moreover,
		the word \(U^2\) is structurally admissible. By Lemma~\ref{lem:component-pumping}, \(\operatorname{ord}(ab)=\infty\).
	\end{proof}
	
	\begin{corollary}
		\label{cor:almost-square-examples}
		If any of the words \(\bar{\ell} q \bar{r} p \bar{\ell} q \bar{r}, \bar{\ell} q \bar{p} r \bar{\ell} q \bar{p}, \ell\bar r\ell\) is exactly interpretable by a pair of class transpositions \((a, b)\), then \(\operatorname{ord}(ab)=\infty\).
	\end{corollary}
	
	\section{A criterion for infinite order}
	\label{sec:criterion}
	
	\begin{lemma}
		\label{lem:four-pairs}
		Set
		\[
		\Phi(u)=
		\bigl(r_3+m_3u,\;r_4+m_4u\bigr)\pmod{\gcd(m_1,m_2)}.
		\]
		If the values of the map \(\Phi\) include three of
		the four pairs
		\[
		(r_1,r_1),\quad
		(r_1,r_2),\quad
		(r_2,r_1),\quad
		(r_2,r_2)
		\pmod{\gcd(m_1,m_2)},
		\]
		then \(\gcd(m_1,m_2)\mid r_1-r_2\).
	\end{lemma}
	
	\begin{proof}
		Among any three of the specified pairs, there are two pairs
		that differ only in the first coordinate and two pairs
		that differ only in the second coordinate.
		
		Therefore, there exist integers \(u,v\) such that
		\[
		m_3(u-v)\equiv\pm(r_1-r_2),
		\quad
		m_4(u-v)\equiv0
		\pmod{\gcd(m_1,m_2)}.
		\]
		Multiplying the first congruence by \(m_4' = \frac{m_4}{\gcd(m_3, m_4)}\) and the second by \(m_3' = \frac{m_3}{\gcd(m_3, m_4)}\),
		we obtain $\gcd(m_1, m_2) \mid m_4'(r_1 - r_2)$.
		
		Similarly, there exist integers \(s,t\) such that
		\[
		m_3(s-t)\equiv0,
		\qquad
		m_4(s-t)\equiv\pm(r_1-r_2)
		\pmod{\gcd(m_1,m_2)},
		\]
		whence \(\gcd(m_1, m_2) \mid m_3'(r_1-r_2)\).
		
		Since \(\gcd(m_3',m_4')=1\), we conclude that \(\gcd(m_1,m_2) \mid r_1-r_2\).
	\end{proof}
	
	\begin{proposition}
		\label{prop:six-letter-obstructions}
		None of the words
		\[
		W_1=q\bar p r\bar\ell q\bar r,\quad
		W_2=q\bar p r\bar q\ell\bar p,\quad
		W_3=q\bar p r\bar q\ell\bar r,\quad
		W_4=\ell\bar p r\bar\ell q\bar r
		\]
		is interpretable by a pair of class transpositions.
	\end{proposition}
	
	\begin{proof}	
		Suppose that one of these words \(w\) is interpretable
		by a pair of class transpositions \((a,b)\), and choose an integer
		solution \((u_1^*,\ldots,u_7^*)\) of the system \(\mathcal S(w)\).
		Set
		\[
		\Phi(u)=\bigl(r_3+m_3u,\;r_4+m_4u\bigr)\pmod{\gcd(m_1,m_2)}.
		\]
		Grouping the equations of the system in pairs and considering them
		modulo \(\gcd(m_1,m_2)
		\), we obtain
		\[
		\renewcommand{\arraystretch}{1.25}
		\begin{array}{c|c|c|c}
			&\Phi(u_2^*)&\Phi(u_4^*)&\Phi(u_6^*)\\ \hline
			W_1
			&(r_2,r_1)&(r_1,r_2)&(r_2,r_2)\\
			W_2
			&(r_2,r_1)&(r_2,r_2)&(r_1,r_1)\\
			W_3
			&(r_2,r_1)&(r_2,r_2)&(r_1,r_2)\\
			W_4
			&(r_1,r_1)&(r_1,r_2)&(r_2,r_2)
		\end{array}
		\pmod{\gcd(m_1,m_2)}.
		\]
		For example, when \(w=W_1\), the first two equations have the form
		\[
		r_2+m_2u_1^*=r_3+m_3u_2^*,
		\qquad
		r_4+m_4u_2^*=r_1+m_1u_3^*,
		\]
		whence \(\Phi(u_2^*)\equiv(r_2,r_1)\pmod{\gcd(m_1,m_2)}\).
		The remaining entries in the table are obtained similarly.
		
		Each row of the table contains three of the four pairs
		specified in Lemma~\ref{lem:four-pairs}. Therefore,
		\(\gcd(m_1,m_2)\mid r_1-r_2\). This contradicts the fact that
		\(a=\tau_{r_1(m_1),r_2(m_2)}\) is a class transposition.
	\end{proof}
	
	For convenience, set
	\[
	\begin{aligned}
		\mathcal F={}&
		\bigcup_{i=1}^{4}\operatorname{Orb}(W_i)
		\cup
		\operatorname{Orb}(\ell\bar r\ell)
		\cup
		\operatorname{Orb}(\ell\bar r p\bar\ell r),\\
		\mathcal E={}&
		\operatorname{Orb}(\bar\ell q\bar r p\bar\ell q\bar r)
		\cup
		\operatorname{Orb}(\bar\ell q\bar p r\bar\ell q\bar p).
	\end{aligned}
	\]
	
	\begin{lemma}
		\label{lem:length-seven-factors}
		Every structurally admissible word \(w\) of length \(7\) contains
		a factor \(v\), possibly equal to \(w\) itself, such that \(v\in \mathcal F \cup \mathcal E\).
	\end{lemma}
	
	\begin{proof}
		It suffices to prove that every structurally admissible word
		of length \(7\) containing no factors from \(\mathcal F\)
		belongs to \(\mathcal E\).
		
		Let \(w\) be such a word. The group $H=\langle\sigma_1,\sigma_2,\sigma_3\rangle$ acts transitively on the alphabet \(\Sigma\).
		The permutations \(\sigma_1,\sigma_2,\sigma_3\) preserve
		structural admissibility, and the set \(\mathcal F\), being
		a union of orbits, is invariant under \(H\).
		Moreover, if \(u\) is a factor of the word \(z\), then
		\(h(u)\) is a factor of \(h(z)\) for each \(h\in H\).
		
		Consequently, there exists \(h\in H\) such that the word \(\widetilde w=h(w)\) begins with \(\bar\ell\), remains structurally admissible,
		and contains no factors from \(\mathcal F\).
		
		From the table of values of \(\pi\), we obtain the table of admissible
		continuations:
		\begin{equation}
			\renewcommand{\arraystretch}{1.2}
			\begin{array}{c|c@{\qquad}c|c}
				\alpha&\text{next letters}
				&\alpha&\text{next letters}\\ \hline
				\ell&\bar r,\bar p
				&\bar\ell&r,q\\
				r&\bar\ell,\bar q
				&\bar r&\ell,p\\
				p&\bar\ell,\bar q
				&\bar p&r,q\\
				q&\bar r,\bar p
				&\bar q&\ell,p
			\end{array}
			\label{eq:admissible-continuations}
		\end{equation}
		
		For \(1\leq n\leq7\), let \(\mathcal B_n\) denote
		the set of structurally admissible words of length \(n\) beginning
		with \(\bar\ell\) and containing no factors from \(\mathcal F\).
		These sets are constructed successively: each word
		is extended by one of the two letters allowed by table
		\eqref{eq:admissible-continuations}, after which the extensions that
		have acquired a forbidden suffix are discarded. It suffices to check suffixes,
		since the original prefix already contains no factors from
		\(\mathcal F\).
		
		Successive enumeration gives
		\[
		\begin{array}{c|rrrrrrr}
			n&1&2&3&4&5&6&7\\ \hline
			|\mathcal B_n|&1&2&3&5&7&5&3.
		\end{array}
		\]
		To verify the last step, we write out
		\[
		\mathcal B_6=\{
		\bar\ell r\bar q p\bar\ell r,
		\bar\ell q\bar r p\bar\ell q,
		\bar\ell q\bar r p\bar q\ell,
		\bar\ell q\bar p r\bar\ell q,
		\bar\ell q\bar p r\bar q\ell
		\}.
		\]
		Each of these words has two structurally admissible
		continuations. Seven of the ten resulting words end in
		forbidden factors. Their membership in the set
		\(\mathcal F\) is verified by the following equalities:
		\[
		\begin{aligned}
			\bar\ell r\bar\ell
			&=\sigma_3(\ell\bar r\ell),\\
			q\bar r p\bar\ell q\bar p
			&=(\rho\circ\sigma_2\circ\sigma_1)(W_2),\\
			q\bar r p\bar q\ell\bar r
			&=(\rho\circ\sigma_2)(W_1),\\
			q\bar r p\bar q\ell\bar p
			&=\sigma_1(W_4),\\
			q\bar p r\bar\ell q\bar r
			&=W_1,\\
			q\bar p r\bar q\ell\bar r
			&=W_3,\\
			q\bar p r\bar q\ell\bar p
			&=W_2.
		\end{aligned}
		\]
		Thus,
		\begin{equation}
			\begin{aligned}
				\mathcal B_7=\{&
				\bar\ell r\bar q p\bar\ell r\bar q,\,
				\bar\ell q\bar r p\bar\ell q\bar r,\,
				\bar\ell q\bar p r\bar\ell q\bar p
				\}.
			\end{aligned}
			\label{eq:length-seven-survivors}
		\end{equation}
		The second and third words in~\eqref{eq:length-seven-survivors} are
		representatives of the two orbits forming \(\mathcal E\).
		The first word also belongs to \(\mathcal E\), since
		\[
		\bar\ell r\bar q p\bar\ell r\bar q
		=
		(\rho\circ\sigma_3\circ\sigma_2)
		(\bar\ell q\bar p r\bar\ell q\bar p).
		\]
		Consequently, \(\widetilde w\in\mathcal E\). Since
		\(\mathcal E\) is a union of orbits and \(\widetilde w=h(w)\),
		we obtain \(w\in\mathcal E\).
		
		Thus, if \(w\) contains no factors from \(\mathcal F\),
		then \(w\) itself is the required factor. Otherwise,
		the required factor belongs to \(\mathcal F\).
	\end{proof}
	
	\begin{theorem}
		\label{thm:length-seven}
		The equality \(\operatorname{ord}(ab) = \infty\) holds if and only if there exists a word that is exactly interpretable by the pair \((a,b)\) and has length \(7\).
	\end{theorem}
	
	\begin{proof}
		Let a structurally admissible word of length \(7\) be exactly
		interpretable by the pair \((a,b)\). By Lemma~\ref{lem:length-seven-factors},
		it contains a factor from \(\mathcal F\cup\mathcal E\), which
		is exactly interpretable by Lemma~\ref{lem:factor}.
		Taking Remark~\ref{rem:orbit-invariance} into account,
		Proposition~\ref{prop:six-letter-obstructions} rules out factors
		from the orbits of the words \(W_1,\ldots,W_4\), while
		Corollaries~\ref{cor:five-letter-word} and~\ref{cor:almost-square-examples}
		give $\operatorname{ord}(ab)=\infty$ in all other cases.
		
		Conversely, suppose that $\operatorname{ord}(ab)=\infty$. By Theorem~\ref{propmain}, the graph \(\Gamma(a,b)\) has
		an infinite component or finite components
		of unbounded size. In both cases, the graph
		\(\Gamma(a,b)\) contains a simple path of \(15\) edges
		beginning and ending with edges of the second type.
		After an orientation and labels are chosen, this path, by
		Proposition~\ref{prop:word-path-correspondence}, determines
		a structurally admissible word of length \(7\) that is exactly
		interpretable by the pair \((a,b)\).
	\end{proof}
	
	\section{The main theorem and sharpness of the bound}
	\label{sec:main-result}
	
	\begin{theorem}
		\label{thm:main}
		If \(\operatorname{ord}(ab)<\infty\), then \(\operatorname{ord}(ab)\mid840\).
	\end{theorem}
	
	\begin{proof}
		By Theorem~\ref{thm:length-seven}, the graph \(\Gamma(a,b)\) contains no
		simple path of \(15\) edges beginning and ending
		with edges of the second type. Indeed, an orientation and labels
		for such a path would, by
		Proposition~\ref{prop:word-path-correspondence}, give
		an exact interpretation of a structurally admissible word of length \(7\).
		
		Let \(C\) be a finite path component. Its
		end vertices are incident to edges of the second type, and edges of the first
		and second types alternate. If \(|C|\geq15\), then the first
		\(15\) edges of the component would form a forbidden simple
		path. Consequently, since the length of a finite path is odd, $|C|\leq13$.
		
		Now let \(C\) be a finite cycle component. Its length
		is divisible by \(4\). If \(|C|\geq16\), then any \(15\)
		consecutive edges of the cycle whose first edge is of the second
		type would form a forbidden simple path. Therefore, \(|C|\leq12\).
		
		By Theorem~\ref{propmain} and Remark~\ref{rem:component-cycles}, a path component of length \(n\) determines
		a cycle of the permutation \(ab\) of length
		\[
		\frac{n+3}{2}\leq8,
		\]
		while a cycle component of length \(n\) determines cycles of length
		\[
		\frac n4\leq3.
		\]
		Consequently, all cycles of the permutation \(ab\) have length at
		most \(8\), and hence
		\[
		\operatorname{ord}(ab)
		\mid\operatorname{lcm}(1,2,3,4,5,6,7,8)
		=840.
		\]
	\end{proof}
	
	\begin{example}
		\label{ex:order-840}
		Consider the pair of class transpositions \((\tau_{0(40),\,18(28)}, \tau_{4(42),\,5(15)})\). These class transpositions are well defined, since \(\gcd(40,28)\nmid18\) and \(\gcd(42,15)\nmid1\). Then
		\[
		\operatorname{ord}(\tau_{0(40),\,18(28)} \cdot \tau_{4(42),\,5(15)})=840.
		\]
		
		This equality was verified by Kohl using the
		\textsf{RCWA} package \cite{RCWA} for \textsf{GAP}. Based on the code
		he provided, the computation can be written as follows:
		
		\begingroup
		\small
		\begin{verbatim}
			gap> LoadPackage("rcwa");;
			gap> CTSparse := true;;
			gap> a := ClassTransposition(18,28,0,40);;
			gap> b := ClassTransposition(5,15,4,42);;
			gap> g := a*b;;
			gap> P := RespectedPartition(g);;
			gap> Order(g);
			840
			gap> lengths := Set(List(Cycles(g,P),Length));
			[ 1, 2, 3, 4, 5, 7, 8 ]
			gap> Lcm(lengths);
			840
		\end{verbatim}
		\endgroup
		
		Here \texttt{RespectedPartition(g)} finds a finite partition
		\(P\) of the set \(\mathbb Z\) into residue classes that
		are permuted by the map \(g\), with
		\(g\) acting affinely on each of these classes. The command \texttt{Order(g)}
		computes the order of the permutation \(g\) itself, while the last two
		commands find the set of cycle lengths of its action on \(P\)
		and their least common multiple. Thus, the order of the product \(ab\) is \(840\),
		which proves that the bound in Theorem~\ref{thm:main} is sharp.
	\end{example}
	
	\begin{question}
		\label{question:finite-orders}
		Which positive divisors of \(840\) occur as orders
		of products of two class transpositions?
	\end{question}
	
	\section*{Acknowledgements}
	
	The work was supported by the Russian Science Foundation, project 24-11-00119, \url{https://rscf.ru/en/project/24-11-00119/}. The author thanks Stefan Kohl for verifying the order of the product
	of class transpositions in the example and for checking the manuscript.
	The author is grateful to his advisor, Valeriy G.~Bardakov, for posing
	the problem and for his guidance and assistance at every stage
	of its solution. The author is also deeply grateful to his wife, Elizaveta Iskra,
	for her help, patience, and understanding, as well as for her
	constant support and encouragement throughout the preparation
	of this work.
	
	\section*{Declaration of generative AI and AI-assisted technologies in the manuscript preparation process}
	
	During the preparation of this work, the author used ChatGPT-5.6 Sol to assist in checking parts of the mathematical arguments, to assist in a computational search for an example attaining the bound \(840\), and to assist with translation and language editing. All AI-assisted mathematical arguments and computationally suggested examples were independently verified by the author; in particular, the order-\(840\) example was independently verified by Stefan Kohl using the \textsf{RCWA} package for \textsf{GAP}. The author reviewed and revised all AI-generated content and takes full responsibility for the correctness, originality, and content of the article.

\end{document}